\documentclass[11pt,reqno,oneside]{amsart}
 \usepackage{verbatim,amsmath,amssymb,cite,epsfig,xspace}

\numberwithin{equation}{section}

\theoremstyle{plain}
\newtheorem{theorem}{Theorem}[section]

\theoremstyle{definition}

\newtheorem{proposition}{Proposition}[section]

\theoremstyle{remark}
\newtheorem{remark}{Remark}[section]

\newcommand{\Real}{\mathbb R}

\newcommand{\F}{{\mathcal F}}

\newcommand{\E}{{\mathcal E}}

\newcommand{\la}{\langle}
\newcommand{\ra}{\rangle}

\newcommand{\half}{\frac{1}{2}}

\newcommand{\eps}{\epsilon}
\newcommand{\ignore}[1]{}
\begin{document}

%%%%%%%%%%%%%
% Topmatter %
%%%%%%%%%%%%%
\title[A Generalized Quasi-Nonlocal Coupling Method with Finite Range Interaction]
{A Generalized Quasi-Nonlocal Atomistic-to-Continuum Coupling Method with Finite Range Interaction}

\author{Xingjie Helen Li and Mitchell Luskin}

%%%%%%%%
% Date %
%%%%%%%%
\date{\today}

%\maketitle { \thispagestyle{empty}

\subjclass[2000]{65Z05,70C20}

\begin{abstract}
The accurate and efficient computation of the deformation of
crystalline solids requires the coupling of atomistic models
near lattice defects such as cracks and dislocations with
coarse-grained models away from the defects. Quasicontinuum
methods utilize a strain energy density derived from the
Cauchy-Born rule for the coarse-grained model.

Several quasicontinuum methods have been proposed to couple the
atomistic model with the Cauchy-Born strain energy density.
The quasi-nonlocal coupling method is easy to implement and
achieves a reasonably accurate coupling for short range
interactions.  In this paper, we give a new formulation of the
quasi-nonlocal method in one space dimension that allows its
extension to arbitrary finite range interactions. We also give
an analysis of the stability and accuracy of a linearization of
our generalized quasi-nonlocal method that holds for strains up
to lattice instabilities.

\end{abstract}

\maketitle

\section{Introduction}
The motivation for multiscale coupling methods is that the accuracy of
a fine scale model is often only needed in localized regions of the computational domain containing defects,
but a coarse-grained continuum model is needed to make efficient the computation of  large enough systems.
The first quasicontinuum approximation was the energy-based quasicontinuum model (denoted QCE \cite{Ortiz:1995a}).
However, there are interfacial forces (called ``ghost forces'') when a uniform strain is
modeled by the QCE energy~\cite{Shenoy:1999a,Dobson:2008a}, or equivalently, the uniform
strain is not an equilibrium solution for the QCE energy (even though the uniform strain is an equilibrium for
the purely atomistic and coarse-grained continuum models). The effect of the ghost force on
the error in computing the deformation and the lattice stability by the QCE approximation has been analyzed
in~\cite{Dobson:2008a,Dobson:2008c,mingyang,doblusort:qce.stab}.
%and investigated in the benchmark test~\cite{Miller:2008}.

Thus, there is a need for a more accurate atomistic-to-continuum coupling
method. The ghost force correction method (GFC) achieves an increased accuracy
by adding a correction to the ghost forces as a dead load during a quasistatic
process~\cite{Shenoy:1999a,Dobson:2008a,dobsonluskin08,Miller:2008,qcf.iterative}.
More accurate coupling can be achieved by a force-based
approximation~\cite{doblusort:qcf.stab,dobs-qcf2,curt03,Miller:2003a},
but the non-conservative equilibrium equations make the iterative
solution and the determination
of lattice stability more challenging~\cite{qcf.iterative}.

An alternative approach is to develop a quasicontinuum energy that
is more accurate than QCE.  The quasi-nonlocal energy (QNL) was the first
quasicontinuum energy without ghost forces in the atomistic-to-continuum interface
for a uniformly strained lattice~\cite{Shimokawa:2004}.
The developers of the QNL method introduced interfacial atoms, which interacted
with the atomistic region using the atomistic model and interacted with the continuum region
using the continuum model.  For a one dimensional chain, the original QNL method is restricted to
next-nearest-neighbor interactions.
In this paper, we formulate the one-dimensional QNL energy in terms of
interactions rather than the energy contributions of individual atoms~\cite{doblusort:qce.stab},
which allows us to generalize the original QNL energy to arbitrary finite range interactions.
A closely related method has been independently proposed
and studied numerically for one and two dimensional problems by Shapeev in~\cite{shapeev}.

We also give an analysis of the
stability and accuracy of a linearization of the generalized quasi-nonlocal method
in the one dimensional case.
The stability and optimal order error analysis for a linearization
of the original QNL model with second nearest-neighbor interaction
was analyzed for the one dimensional case in \cite{Dobson:2008b}.
A nonlinear {\it a priori} and {\it a posteriori} error analysis
for the QNL model with second nearest-neighbor interaction in one dimension was
given in ~\cite{ortner:qnl1d}.

In Section~\ref{lqc}, we describe the coarse-grained local QC energy and the
motivation for the development of more accurate QC energies.

In Section~\ref{notation}, we present the notation used in this paper. We define the displacement space $\mathcal{U}$
and the deformation space $\mathcal{Y}_{F}$. We then introduce the norms we will use
to estimate the consistency error and the displacement gradient error.

In Section~\ref{gqnl}, we introduce the general QNL energy with
finite range interaction and give an expression of the energy
and equilibrium equations associated with the model.

In Section~\ref{sharp}, we give sharp stability estimates for both
the fully atomistic model and the generalized QNL model with respect to a uniform strain.
Sharp stability estimates are necessary to determine whether quasicontinuum methods
(or other coupling methods) are accurate near instabilities
such as defect formation or crack propagation~\cite{doblusort:qce.stab,doblusort:qcf.stab}.
Similar stability estimates of the
fully atomistic and fully local quasi-continuum (QCL) models have been studied by discrete Fourier analysis
in \cite{HudsOrt:a}.

In section~\ref{convergence}, we study the convergence rate of the generalized QNL model.
We compare the equilibrium solution of the generalized QNL model
with that of the fully atomistic model, and we
use the negative norm estimation method \cite{dobs-qcf2} to obtain an optimal rate of
convergence of the strain error.  The error estimate depends only on the smoothness
of the strain in the continuum region and holds near lattice instabilities,
thus demonstrating that the generalized
QNL method can give a small error if defects are captured in the atomistic region.

\section{The Local QC Approximation and its Error}\label{lqc}
In this section, we first briefly describe the $1$D atomistic chain model.
We give the fully atomistic energy, $\mathcal{E}^{a}(\mathbf{Y}),$ and then use linear interpolation
to derive the coarse-grained local QC energy, $\mathcal{E}^{qc}(\mathbf{Y}).$
We next compute and analyze the difference between these two energies. We then rescale the model
and conclude that the coarse-grained local QC energy is formally a second order approximation
(where the small parameter is lattice spacing scaled by
the size of the macroscopic domain) to the fully atomistic energy when the strain gradient
is small.

For simplicity, we assume that the infinite atomistic chain with positions $y_{\ell}<y_{\ell+1}$
has period $2N,$ that is,
\[
y_{\ell+1+2N}-y_{\ell+2N}=y_{\ell+1}-y_{\ell}\quad \forall \ell\in
\mathbb{Z},
\]
where $\mathbb{Z}$ is the set of integers.
We compute the energies of
interaction using a potential $\tilde\phi:(0,\infty]\rightarrow\Real$ which
we assume satisfies
\begin{enumerate}
\item $\tilde\phi \in C^4((0,\infty], \Real);$
\item there exists $r^*$ so that $\tilde{\phi}''(r) \ge 0$ for all $r
    \in (0, r^*]$, and $\tilde\phi''(r) < 0$ for all $r \in
    (r^*,\infty);$
\item
$\tilde\phi$ and its derivatives decay at infinity.
\end{enumerate}

Examples of commonly used interaction potentials which satisfy the
above conditions include the Lennard-Jones potential and the Morse
potential.
The total stored atomistic energy per period
for the interaction potential, $\tilde\phi(r),$ with up to $s$th nearest-neighbor interactions is given by
\begin{align}%\label{EnergyMultiNbhd1}
\begin{split}
  {\E}^{a}(\mathbf{y})&:=%\sum_{\ell=-N+1}^{N}\tilde\phi(y_{\ell}-y_{\ell-1})+
\sum_{k=1}^{s}\sum_{\ell=-N+1}^{N}\tilde\phi(y_{\ell+k-1}-y_{\ell-1}).
\end{split}
\end{align}
We note that $\mathcal{E}^{a}(\mathbf{y})$ can be rewritten as
\begin{align}\label{EnergyMultiNbhd1}
\begin{split}
 \mathcal{E}^{a}(\mathbf{y})&=\sum_{k=1}^{s}\sum_{\ell=-N+1}^{N}\tilde\phi\big(k(y_{\ell}-y_{\ell-1})\big)\\
&\quad+\sum_{k=2}^{s}\sum_{\ell=-N+1}^{N}\left\{\tilde\phi(y_{\ell+k-1}-y_{\ell-1})
-\frac{1}{k}\sum_{t=0}^{k-1}\tilde\phi\big(k(y_{\ell+t}-y_{\ell+t-1})\big)\right\}\\
&=\sum_{\ell=-N+1}^{N}\tilde\phi_{cb}(y_{\ell}-y_{\ell-1})\\
&\quad+\sum_{k=2}^{s}\sum_{\ell=-N+1}^{N}\left\{\tilde\phi(y_{\ell+k-1}-y_{\ell-1})
-\frac{1}{k}\sum_{t=0}^{k-1}\tilde\phi\big(k(y_{\ell+t}-y_{\ell+t-1})\big)\right\}
\end{split}
\end{align}
where
\[
  \tilde\phi_{cb}(\tilde r) :=\sum_{k=1}^{s}\tilde\phi(k\,\tilde r)
\]
is  the
\textit{Cauchy-Born}
energy density \cite{Ortiz:1995a,Dobson:2008b}.

Intuitively, we can reduce the total amount of computation by first
choosing $2M$, $M<<N$, representative atoms in one period, linearly
interpolating the positions of the remaining atoms, and then computing the
total atomistic energy by using the interpolated positions. More precisely,
we introduce the representative atoms with positions $Y_{j}$ such
that
\begin{align*}
 Y_{j}=y_{\ell_{j}},\quad j=-M,\dots,M,
\end{align*}
where the subindex $\ell_j$ for $j=-M,\dots,M$ satisfies
\[
-N=\ell_{-M}<\dots<\ell_{j-1}<\ell_{j}<\dots<\ell_{M}=N.
\]
The representative atoms thus satisfy
\[
Y_{-M}<\dots<Y_{j-1}<Y_j<\dots<Y_M.
\]
We
denote the number of atoms between $Y_{j-1}$ and $Y_{j}$ as $\tilde \nu_{j},$ and
the distance separating the
$\tilde \nu_{j}$ equally spaced atoms between $Y_{j-1}$ and $Y_{j}$
as $\tilde r_j,$ that is,
\[
\tilde \nu_{j}:=\ell_{j}-\ell_{j-1}\quad\text{ and }\quad
\tilde r_j=\left(Y_{j}-Y_{j-1}\right)/\tilde \nu_{j}\quad\text{for }
j=-M+1,\dots,M.
\]
Then the positions of the atoms between $Y_{j-1}$
and $Y_j$ can be approximated by
\begin{equation}\label{int}
 y_{\ell_{j-1}+i}=y_{\ell_{j}-(\tilde \nu_{j}-i)}=Y_{j-1}+i\,\tilde r_{j}=
 \frac{\tilde \nu_{j}-i}{\tilde \nu_{j}}Y_{j-1}+\frac{i}{\tilde \nu_{j}}Y_{j}, \qquad 0\le i\le\tilde \nu_{j}.
\end{equation}
We thus have that
\[
Y_{j-1}=y_{\ell_{j-1}}<y_{\ell_{j-1}+1}<\dots<y_{\ell_{j}-(\tilde \nu_{j}-1)}<y_{\ell_{j}}=Y_j,\qquad 0\le i\le\tilde \nu_{j}.
\]

We further assume that each $\tilde\nu_{j}$ is sufficiently large, i.e.,
$\tilde\nu_{j}\ge s$ for $j=-M+1,\dots,M$, so that the
the span of a bond can cross at most one interface.
We can then observe for the atomistic deformations given by the interpolation ~\eqref{int} that
\begin{align*}
 &\tilde\phi(y_{\ell+k-1}-y_{\ell-1})-\frac{1}{k}\sum_{t=0}^{k-1}\tilde\phi(k(y_{\ell+t}-y_{\ell+t-1}))\\
&\quad=\begin{cases}
0, & \,\ell_{j-1}<\ell<\ell_{j}-(k-2),\\
%\tilde\phi\big(\tilde r_{j-1}+(k-1)\,\tilde r_{j}\big)-\frac{1}{k}\tilde\phi(k\,\tilde r_{j-1})-\frac{k-1}{k}\tilde\phi(k\,
%\tilde r_{j}),& \, \ell=\ell_{j-1},\\
\tilde\phi\big(\,(p+1)\tilde r_{j}+(k-p-1)\,\tilde r_{j+1}\,\big)-\frac{p+1}{k}\tilde\phi(k\,
\tilde r_{j})
-\frac{k-p-1}{k}\tilde\phi(k\,\tilde r_{j+1}), & \, \ell=\ell_{j}-p,\ 0\le p\le k-2.\\
\end{cases}
\end{align*}
Rearranging the terms in \eqref{EnergyMultiNbhd1} and applying
the equalities above, we can rewrite the total atomistic energy
$\mathcal{E}^{a}(\mathbf{Y})$ as the sum of a coarse-grained local QC energy,
$\mathcal{ E}^{qc}(\mathbf{Y}),$
and an interfacial energy:
\begin{align}\label{EnergyMultiNbhd2}
\begin{split}
    \mathcal{E}^{a}(\mathbf{Y})&=\sum_{k=1}^{s}\sum_{\ell=-N+1}^{N}\tilde\phi\big(k(y_{\ell}-y_{\ell-1})\big)\\
&\quad+\sum_{k=2}^{s}\sum_{\ell=-N+1}^{N}\left\{\tilde\phi(y_{\ell+k-1}-y_{\ell-1})
-\frac{1}{k}\sum_{t=0}^{k-1}\tilde\phi\big(k(y_{\ell+t}-y_{\ell+t-1})\big)\right\}\\
%&=\sum_{k=1}^{s}\sum_{i=-M}^{M}\tilde\phi\big(k(y_{i+1}-y_{i})\big)\\
%&\quad
%+\sum_{k=2}^{s}\left\{\sum_{i=-M+1}^{M}\tilde\phi(y_{i+k-1}-y_{i-1})
%-\frac{1}{k}\sum_{t=0}^{k-1}\tilde\phi\big(k(y_{i+t}-y_{i+t-1})\big)\right\}\\
&=\mathcal{ E}^{qc}(\mathbf{Y})
+\sum_{j=-M+1}^{M}{\mathcal{P}}_{j},
\end{split}
\end{align}
where the coarse-grained local QC model is
\begin{equation}\label{localQCenergy}
 \mathcal{ E}^{qc}(\mathbf{Y}):=\sum_{j=-M+1}^{M}\tilde \nu_{j}\,\tilde\phi_{cb}(\tilde r_{j}),
\end{equation}
and
where the interfacial energy is
\begin{align*}
 {\mathcal{P}}_{j}&:=\sum_{k=2}^{s}
\sum_{p=1}^{k-1}\Big\{\tilde\phi\left(
p\tilde r_{j}+(k-p)\tilde r_{j+1}\right)
-\frac{p}{k}\tilde\phi(k\,\tilde r_{j})-\frac{k-p}{k}\tilde\phi(k\,\tilde r_{j+1})\Big\},\quad
j=-M+1,\dots, M.
\end{align*}
We note that we have used the
periodicity for the representative atoms chosen in one period. Therefore, we have
that $\tilde r_{M+1}=\tilde r_{-M+1}$ and
the definition of $\mathcal{P}_{M}$ makes sense.
%**********************************

To understand the difference between the atomistic energy
$\mathcal{E}^{a}(\mathbf{Y})$ and the coarse-grained local QC energy
$\mathcal{ E}^{qc}(\mathbf{Y})$, we first evaluate the interfacial
energy terms, namely ${\mathcal{P}}_{j},\, j=-M+1,\dots,M$. Using
the Lagrange form of the Taylor expansion of degree $2$ at
$\frac{k}{2}(\tilde r_{j+1}+\tilde r_{j})$, we can obtain that
$\mathcal{P}_{j}$ is given by
\begin{align}\label{EnergyMultiNbhd3}
\begin{split}
 \mathcal{P}_{j} &= \sum_{k=2}^{s}
\left\{\frac{-k^3+k}{12} \tilde\phi''\bigg(
k\big(\eta_{jk}\tilde r_{j}+(1-\eta_{jk})\tilde r_{j+1}\big)\bigg) (\tilde r_{j+1}-\tilde r_{j})^2\right\},
 \quad j=-M+1,\dots, M,
\end{split}
\end{align}
and
$\eta_{jk}$ is a scalar satifying $0\le\eta_{jk}\le 1.$ %in the interval $[\tilde r_{j},\tilde r_{j+1}]$
%(or the interval $[\tilde r_{j+1},\tilde r_{j}]$).
\subsection{Scaled Models}
We next consider
a scaled version of the atomistic and local QC energies. Thus, we
define the two scaled interaction potentials $\phi(r)$ and $\phi_{cb}(r)$ by
\[
\phi(r)=\frac{1}{\epsilon} \tilde\phi(r \epsilon)
\quad\text{and}\quad
\phi_{cb}(r)=\frac{1}{\epsilon} \tilde\phi_{cb}(r \epsilon)
\]
where $\epsilon>0$ scales the reference lattice.
We can now convert
\eqref{EnergyMultiNbhd3} to the scaled form:
\begin{align}\label{DimLess1}
\begin{split}
 \mathcal{P}_{j}
&= \epsilon \sum_{k=2}^{s}\left\{ \frac{-k^3+k}{12}  \phi''
\bigg(
k\big(\eta_{jk} r_{j}+(1-\eta_{jk}) r_{j+1}\big)\bigg)
%\right\}
\,(r_{j}-r_{j+1})^2
\right\}
\end{split}
\end{align}
where $r_j:=\tilde r_j/\eps$ and
$\eta_{jk}$ is a scalar satifying $0\le\eta_{jk}\le 1.$

Therefore, the total atomistic energy \eqref{EnergyMultiNbhd2} has
the scaled form
\begin{align}\label{DimlessEnergy1}
\begin{split}
{\E}^{a}(\mathbf{Y})&=\sum_{j=-M+1}^{M}
\tilde \nu_{j}\,\tilde\phi_{cb}(\,\tilde r_{j})+\sum_{j=-M+1}^{M}{\mathcal{P}}_{j}
\\
&=\sum_{j=-M+1}^{M}\nu_{j}\,\phi_{cb}(r_{j})+\epsilon\sum_{j=-M+1}^{M}(r_{j}-r_{j+1})^2
\sum_{k=2}^{s}\left\{\frac{-k^3+k}{12} \phi''\bigg(
k\big(\eta_{jk} r_{j}+(1-\eta_{jk}) r_{j+1}\big)\bigg) \right\}\\
%&= \mathcal{
%E}^{qc}(\mathbf{Y})-\frac{\epsilon}{2}\sum_{j=-M+1}^{M}(r_{j}-r_{j+1})^2
%\sum_{k=2}^{s}\left\{(k-1)^2 \phi''\left(
%\frac{k}{2}(r_{j}+r_{j+1})\right) \right\}
%\\
%&= \mathcal{
%E}^{qc}(\mathbf{Y})-\frac{\epsilon}{2}\sum_{j=-M+1}^{M}(Y_{j}'-Y_{j+1}')^2
%\sum_{k=2}^{s}\left\{(k-1)^2 \phi''\left(
%\frac{k}{2}(Y_{j}'+Y_{j+1}')\right) \right\}
%\\
%&= \mathcal{
%E}^{qc}(\mathbf{Y})-\sum_{j=-M+1}^{M}
%H_{j+1}\left[\frac{\epsilon H_{j+1}}{2} (Y_{j+1}'')^2\right]
%\sum_{k=2}^{s}\left\{(k-1)^2 \phi''\left(
%\frac{k}{2}(Y_{j}'+Y_{j+1}')\right) \right\},
&= \mathcal{E}^{qc}(\mathbf{Y})+\sum_{j=-M+1}^{M}
H_{j+1}C_{j}\left\{\epsilon H_{j+1} (Y_{j+1}'')^2\right\}
\end{split}
\end{align}
where
\[
 C_{j}:=\sum_{k=2}^{s}\left\{ \frac{-k^3+k}{12} \phi''\bigg(
k\big(\eta_{jk} Y_{j}'+(1-\eta_{jk}) Y_{j+1}'\big)\bigg)\right\},
\]
and
\begin{alignat*}{2}
X_j&:=\eps \ell_j, & H_j&:=(X_j-X_{j-2})/2, \\
\nu_j &:=\eps\tilde \nu_j, &Y_j'&:=\frac{Y_j-Y_{j-1}}{X_j-X_{j-1}}%=\frac{Y_j-Y_{j-1}}{\eps \nu_j}
=r_j,\\
Y_{j}''&:=\frac{Y_j'-Y_{j-1}'}{H_j}.&&
\end{alignat*}

\begin{proposition}
 The coarse-grained local QC energy $\mathcal{E}^{qc}(\mathbf{Y})$ is formally a second order
approximation to the fully atomistic energy $\mathcal{E}^{a}(\mathbf{Y})$, or more precisely,
\begin{equation*}\label{DimlessEnergy1a}
\begin{split}
 {\E}^{a}(\mathbf{Y})&=\mathcal{E}^{qc}(\mathbf{Y})+\sum_{j=-M+1}^{M}
H_{j+1}C_{j}\left\{\epsilon H_{j+1} (Y_{j+1}'')^2\right\}
%+O\left(\sum_{j=-M+1}^{M}
%\epsilon H_{j+1}^4(Y_{j+1}'')^4\right)
,\\
\end{split}
\end{equation*}
where $H=\max_j H_j.$
\end{proposition}

\section{Notation}
\label{notation}
Before further discussing the scaled models, we present the notation used in this paper.
We define the reference lattice
\[
\eps \mathbb{Z}:= \{\eps\ell : \ell\in\mathbb{Z}\},
\]
 where $\eps >0$
scales the reference atomic spacing and $\mathbb{Z}$ is the set of integers.  We then deform
the reference lattice $\eps \mathbb{Z}$
uniformly into the lattice
\[
F\eps \mathbb{Z}:= \{F\eps\ell : \ell\in\mathbb{Z}\},
\]
where
$F >0$ is the macroscopic deformation gradient, and we define the corresponding deformation
$\mathbf{y}_{F}$ by
\[
(\mathbf{y}_{F})_{\ell}:=F\eps \ell  \quad \text{for } -\infty
<\ell<\infty.
\]
For simplicity, we consider the space $\mathcal{U}$ of $2N$-periodic zero mean displacements
$\mathbf{u}=(u_{\ell})_{\ell \in \mathbb{Z}}$ from $\mathbf{y}_{F}$ given by
\[
\mathcal{U}:=\{\mathbf{u} :
u_{\ell+2N}=u_{\ell} \quad \text{for}\,\ell\in \mathbb{Z},\,
\text{and}\sum_{\ell=-N+1}^{N}u_{\ell}=0\},
\]
and we thus admit deformations $\mathbf{y}$ from the space
\[
\mathcal{Y}_{F}:=\{\mathbf{y}:
\mathbf{y}=\mathbf{y}_{F}+\mathbf{u}\text{ for some }\mathbf{u}\in
\mathcal{U}\}.
\]
We set $\eps=1/N$ throughout so that the reference length of the
periodic domain is fixed.

 We
define the discrete differentiation operator, $D\mathbf{u}$, on
periodic displacements by
\[
(D\mathbf{u})_{\ell}:=\frac{u_{\ell}-u_{\ell-1}}{\epsilon}, \quad
-\infty<\ell<\infty.
\]
We note that $\left(D\mathbf{u}\right)_{\ell}$ is also $2N$-periodic in
$\ell$ and satisfies the mean zero condition. We will denote
$\left(D\mathbf{u}\right)_{\ell}$ by $Du_{\ell}$.
 Then, we define
\begin{align*}
\left(D^{(2)}\mathbf{u}\right)_{\ell}:=\frac{Du_{\ell}-Du_{\ell-1}}{\epsilon},
\qquad -\infty<\ell<\infty.
\end{align*}
We can define $\left(D^{(3)}\mathbf{u}\right)_{\ell}$ and
$\left(D^{(4)}\mathbf{u}\right)_{\ell}$ in a similar way. To make the formulas concise,
we also denote $Du_{\ell}$ by $u'_{\ell}$, $D^{(2)}u_{\ell}$ by $u''_{\ell}$, etc.,
when there is no confusion in the expressions.

%Now we
%define various discrete norms on the space $\mathcal{U}$.
For a displacement $\mathbf{u}\in \mathcal{U}$, we define the discrete
$\ell_{\epsilon}^{2}$ norm by
\begin{align*}
\|\mathbf{u}\|_{\ell_{\epsilon}^{2}}&:= \left( \epsilon
\sum_{\ell=-N+1}^{N}|u_{\ell}|^{2}\right)^{1/2}%,\quad 1\le
%p<\infty,\\
%\|\mathbf{u}\|_{\ell_{\epsilon}^{\infty}}&:=
% \max_{-N+1\le \ell \le N}|u_{\ell}|
.
\end{align*}
The inner product associated with the
$\ell_{\epsilon}^{2}$ norm is
\[
\langle\mathbf{v},\mathbf{w}\rangle:=\epsilon \sum_{\ell=-N+1}^{N} v_{\ell}w_{\ell}
\quad \text{for all }\mathbf{v},\mathbf{w}\in \mathcal{U}.
\]

\section{The Generalized Quasi-Nonlocal Approximation.}\label{gqnl}
From \eqref{DimlessEnergy1}, we find that the difference between the
atomistic energy and the coarse-grained local QC energy is formally of second
order, i.e., $\max_{j} O(\eps H_{j})$. However, when the strain gradient is large
in some regions, namely where
$\epsilon H_{j} (Y_{j}'')^2 $ is large, this error can be unacceptable. Hence, to
maintain both accuracy and efficiency, we should use the atomistic model
where the strain gradient is large and the local QC model where the
strain gradient is moderate.  In this section, we propose a hybrid atomistic-continuum
coupling model that extends the quasi-nonlocal model~\cite{Shimokawa:2004} to include finite-range interactions.
To simplify our analysis, we
will propose and study our quasi-nonlocal model without coarsening,
although the local QC region can be coarse-grained.

%%%%%%%%%%%%%%%%%%%%%%%%%%%%%%%%%%%%%%%%%%%%%%%%%%%

According to our assumption in Section~\ref{lqc}, the total stored energy
of the atomistic model per period (that includes up to the $s$th nearest neighbor pair
interactions) in dimensionless form is
\begin{equation}\label{AtomEnerGen01}
 \mathcal{E}^{a}(\mathbf{y})=\eps \sum_{\ell=-N+1}^{N}\sum_{k=1}^{s} \phi\left( \sum_{j=0}^{k-1}y'_{\ell+j}\right).
\end{equation}
Here, the atomistic energy is a sum over the contributions from each
bond and we can also rewrite it in terms of energy contributions of each atom,
\[
\mathcal{E}^{a}(\mathbf{y})=\epsilon
\sum_{\ell=-N+1}^{N}E^{a}_{\ell}(\mathbf{y}),\quad\text{where}
\]
\[
E_{\ell}^{a}(\mathbf{y}):=\half\sum_{k=1}^{s}\phi(\sum_{j=1}^{k}y'_{\ell+j})
+\half\sum_{k=1}^{s}\phi(\sum_{j=0}^{k-1}y'_{\ell-j}).
\]

If $\mathbf{y}$ is "smooth" near $y_{\ell},$ i.e., $y'_{\ell+j}$ and $y'_{\ell-j}$ vary slowly near $y_{\ell}$, then
we can accurately approximate the distance between $k$th nearest-neighbors of $y_{\ell}$ by that of
first nearest-neighbors to approximate $E_{\ell}^{a}(\mathbf{y})$
by $E_{\ell}^{c}(\mathbf{y})$, where
\begin{align*}
E_{\ell}^{c}(\mathbf{y}):=
\half\sum_{k=1}^{s}\big[\phi(k\,y'_{\ell})+\phi(k\,y'_{\ell+1})\big]
=\half\left[\phi_{cb}(y'_{\ell})+\phi_{cb}(y'_{\ell+1})\right].
\end{align*}

If $y'_{\ell}$ is ``smooth" outside of a region $\{-K,\dots,K\}$, where $1<K<N$,
then the original quasicontinuum energy (denoted QCE)\cite{Ortiz:1995a}
 uses the atomistic energy $E_{\ell}^{a}$ in the atomistic region
$\mathcal{A}:=\{-K,\dots,K\}$ and the local QC energy $E_{\ell}^{c}$ in the continuum
region $\mathcal{C}:=\{-N+1,\dots,N\}\setminus \mathcal{A}$ to obtain
\[
\mathcal{E}^{qce}(\mathbf{y}):=
\epsilon\sum_{\ell=-N+1}^{-K+1}E^{c}_{\ell}(\mathbf{y})
+\epsilon\sum_{\ell=-K}^{K}E^{a}_{\ell}(\mathbf{y})
+\epsilon\sum_{\ell=K+1}^{N}E^{c}_{\ell}(\mathbf{y}).
\]
%%%%%%%%%
Although the idea of the QCE method is simple and appealing, there
are interfacial forces (called ghost forces) even for uniform strain
\cite{Miller:2003a,Dobson:2008a}.  The subsequent low order of consistency
for the QCE method has been analyzed in
\cite{Dobson:2008c,mingyang,doblusort:qce.stab}.

The first quasicontinuum energy without a ghost force, the quasi-nonlocal energy (QNL), was proposed in
~\cite{Shimokawa:2004}.  The QNL energy introduced in~\cite{Shimokawa:2004} was restricted to
next-nearest neighbor interactions which in 1D are next-nearest interactions.
By understanding this method in terms of interactions rather than the
energy contributions of ``quasi-nonlocal'' atoms, we have been able to extend the QNL
energy (for pair potentials) to arbitrarily finite-range interactions while maintaining the
uniform strain as an equilibrium (that is, there are no ghost forces).
A generalization of the QNL energy to finite-range interactions from the point view of atoms was proposed in
\cite{E:2006}, but the construction requires the solution of large systems of linear equations
that have so far not
permitted a feasible general implementation.  The interaction-based approach that we give here has been generalized
to two space dimensions in \cite{shapeev}.

In our QNL energy, the nearest-neighbor
interactions are left unchanged. A $k$th-nearest neighbor interaction
$\phi\left( \sum_{j=0}^{k-1}y'_{\ell+j}\right)$ where $k \ge 2$ is
left unchanged if at least one of the atoms $\ell+j: j=-1,\dots,k-1$
belong to the atomistic region and is replaced by a Cauchy-Born
approximation
\begin{equation}\label{cb}
 \phi\left( \sum_{j=0}^{k-1}y'_{\ell+j}\right)\approx \frac{1}{k}\sum_{j=0}^{k-1}\phi\left( k\,y'_{\ell+j}\right)
\end{equation}
if all atoms $\ell+j: j=-1,\dots,k-1$ belong to the continuum region. We define
\begin{equation}\label{ac}
\begin{split}
\mathcal{A}_{qnl}(k)&:=\{-K-k+1,-K-k+2,\dots,K,K+1\},\\
\mathcal{C}_{qnl}(k)
&:=\left\{ -N+1,\dots,-K-k\right\}\bigcup\left\{ K+2,\dots,N \right\}
\end{split}
\end{equation}
for $k=2,\dots,s.$
Then the our generalized
QNL energy is given by
\begin{align}\label{00QNLEnerGen}
\begin{split}
\mathcal{E}^{qnl}(\mathbf{y}):=&\eps \sum_{\ell=-N+1}^{N}
\phi(y'_{\ell}) +\eps\sum_{k=2}^{s}\sum_{\ell \in A_{qnl}(k)}
\phi\left(\sum_{j=0}^{k-1}y'_{\ell+j}\right)%\\
%&\quad
+\eps\sum_{k=2}^{s}\sum_{\ell \in C_{qnl}(k)}
\frac{1}{k}\sum_{j=0}^{k-1}\phi\left(ky'_{\ell+j}\right).%-\epsilon\sum_{\ell=-N+1}^{N}f_{\ell}y_{\ell}.
\end{split}
\end{align}

Since the forces at the atoms for the Cauchy-Born
approximation~\eqref{cb} are unchanged for uniform strain, the QNL
energy does not have ghost forces.
\begin{proposition}
The QNL energy defined in \eqref{00QNLEnerGen} is consistent under a
uniform deformation, i.e., it does not have a ghost force.
\end{proposition}
\begin{proof}
The force at atom $\ell$ is defined to be
$-\frac{\partial\mathcal{E}^{qnl}(\mathbf{y})}{\partial y_{\ell}}$, so from
 \eqref{00QNLEnerGen}, we only need to verify that the
approximation \eqref{cb} is consistent under a uniform deformation.
We note that
\[
\frac{\partial}{\partial
y_{\ell+m}}\phi\left(\sum_{j=0}^{k-1}y'_{\ell+j}\right)\bigg|_{\mathbf{y}=\mathbf{y}_{F}}
= \frac{\partial}{\partial y_{\ell+m}}\left[\frac{1}{k}
\sum_{j=0}^{k-1}\phi\left(ky'_{\ell+j}\right)\right]\bigg|_{\mathbf{y}=\mathbf{y}_{F}},
\qquad\forall m=-1,\dots,k.
\]
Therefore, the QNL energy is consistent. %\quad\qed
\end{proof}
%\newline

 We note that the negative of the QNL forces for a general deformation $\mathbf{y}^{qnl}\in
\mathcal{Y}_{F}$ is given by
\begin{align}\label{QNLSolEq001}
\begin{split}
\la\delta\mathcal{E}^{qnl}(\mathbf{y}^{qnl}),\mathbf{w}\ra
&= \eps \sum_{\ell=-N+1}^{N} \phi'(y'_{\ell})w'_{\ell}
+\eps\sum_{k=2}^{s}\,\sum_{\ell \in A_{qnl}(k)} \phi'\left( \sum_{t=0}^{k-1}y'_{\ell+j}\right)
\left( \sum_{j=0}^{k-1}w'_{\ell+j}\right)\\
&\qquad +\eps \sum_{k=2}^{s}\,\sum_{\ell \in
C_{qnl}(k)}\sum_{j=0}^{k-1}\phi'\left( k\,y'_{\ell+j}\right)
w'_{\ell+j}\quad \forall \mathbf{w}\in \mathcal{U}.
\end{split}
\end{align}
%where $\mathbf{f}[\mathbf{w}]=\epsilon\sum_{\ell=-N+1}^{N}f_{\ell}w_{\ell}$ is a given dead load.%%%%%%

%%%%%%%%%%%%%%%%%%%%
\section{Sharp Stability Analyses of Atomistic and QNL model.}\label{sharp}
A sharp stability analysis for both the atomistic and QNL models are
needed to determine whether the QNL approximation is accurate
for strains near the limits of lattice stability.
 In this section, we will thus give
stability analyses for both models. In order to provide clear
statements and proofs, we first apply a similar method given in
\cite{doblusort:qce.stab} to a third nearest-neighbor interaction
range problem and then generalize the results to the finite range
case.

\subsection{Atomistic Model with Third Nearest-Neighbor
Interaction Range}The total energy for the atomistic
model given by a third nearest-neighbor interaction without external
forces is
\begin{equation}\label{AtomEner3rd}
\mathcal{E}^{a}(\mathbf{y})= \eps
\sum_{\ell=-N+1}^{N}\phi(y'_{\ell})+
\phi(y'_{\ell}+y'_{\ell+1})+\phi(y'_{\ell}+y'_{\ell+1}+y'_{\ell+2}).
\end{equation}
It is easy to see that the uniform deformation $\mathbf{y}_{F}$ is an {\it equilibrium} of
the atomistic model, that is,
\[
\la\delta\mathcal{E}^{a}(\mathbf{y}_F),\mathbf{w}\ra
=0\quad \forall \mathbf{w}\in \mathcal{U}.
\]
We will say that the equilibrium
$\mathbf{y}_{F}$ is {\em stable} for the atomistic model if
$\delta^{2}\mathcal{E}^{a}(\mathbf{y}_{F})$ is positive definite, that is,
if
\begin{align}\label{sharpStableAto1}
\la\delta^{2}\mathcal{E}^{a}(\mathbf{y}_{F})\mathbf{u},\mathbf{u}\ra&=\eps\sum_{\ell=-N+1}^{N}\big(\phi''_{F}|u'_{\ell}|^2
+\phi''_{2F}|u'_{\ell}+u'_{\ell+1}|^2  \big.\\
& \qquad\quad\big. +
\phi''_{3F}|u'_{\ell}+u'_{\ell+1}+u'_{\ell+2}|^2\big) > 0 \quad
\forall \mathbf{u}\in \mathcal{U}\setminus \{\mathbf{0}\}\nonumber
\end{align}
where
\begin{displaymath}
  \phi_F'' := \phi''(F),\ \phi_{2F}'' := \phi''(2F),\dots,
  \phi_{sF}'' := \phi''(sF).
\end{displaymath}

We observe that
\begin{align}\label{Equa1}
\begin{split}
|u'_{\ell}+u'_{\ell+1}|^2&=2|u'_{\ell}|^2+2|u'_{\ell+1}|^2-|u'_{\ell+1}-u'_{\ell}|^2,\\
|u'_{\ell}+u'_{\ell+1}+u'_{\ell+2}|^2&=3|u'_{\ell}|^2+3|u'_{\ell+1}|^2+3|u'_{\ell+2}|^2\\
&\quad -3\eps^2|u''_{\ell+1}|^2 - 3\eps^2|u''_{\ell+2}|^2 +\eps^2
|u''_{\ell+2}-u''_{\ell+1}|^2.
\end{split}
\end{align}
Because of the periodicity of $u'_{\ell},\,u''_{\ell}$ in $\ell$,
 we note that the sum of $u'_{\ell+1}$ from $-N+1$ to $N$ is equal to the sum of $u'_{\ell}$ and the sum of
$u''_{\ell+1} $ is equal to the sum of $u''_{\ell}$, etc.
So we can rewrite \eqref{sharpStableAto1} and get the lower bound
\begin{align}\label{sharpStableAto2}
\begin{split}
\la\delta^2\mathcal{E}^{a}(\mathbf{y}_{F})\mathbf{u},\mathbf{u}\ra&=
\eps\sum_{\ell=-N+1}^{N}\big(\phi''_{F}+4\phi''_{2F}+9\phi''_{3F}\big)|u'_{\ell}|^2
-\eps\sum_{\ell=-N+1}^{N}\big(\eps^2\phi''_{2F}+6\eps^2\phi''_{3F}\big)|u''_{\ell}|^2\\
&\qquad
+\eps\sum_{\ell=-N+1}^{N}\big(\eps^2\phi''_{3F}\big)|u''_{\ell+1}-u''_{\ell}|^2\\
&=
\big(\phi''_{F}+4\phi''_{2F}+9\phi''_{3F}\big)\|\mathbf{u}'\|_{\ell_{\eps}^{2}}^2
-\big(\eps^2\phi''_{2F}+2\eps^2\phi''_{3F}\big)\|\mathbf{u}''\|_{\ell_{\eps}^{2}}^2\\
&\qquad-\eps\big(\eps^2\phi''_{3F}\big)\sum_{\ell=-N+1}^{N}\left[4|u''_{\ell}|^2-|u''_{\ell+1}-u''_{\ell}|^2\right]\\
%&\ge
%\big(\phi''_{F}+4\phi''_{2F}+9\phi''_{3F}\big)\|\mathbf{u}'\|_{\ell_{\eps}^{2}}^2
%-\big(\eps^2\phi''_{2F}+6\eps^2\phi''_{3F}\big)\|\mathbf{u}''\|_{\ell_{\eps}^{2}}^2\\
%&\qquad+\eps\sum_{\ell=-N+1}^{N}\big(\eps^2\phi''_{3F}\big)\big(2|u''_{\ell+1}|^2+2|u''_{\ell}|^2\big)\\
&\ge
\big(\phi''_{F}+4\phi''_{2F}+9\phi''_{3F}\big)\|\mathbf{u}'\|_{\ell_{\eps}^{2}}^2
-\eps^2\big(\phi''_{2F}+2\phi''_{3F}\big)\|\mathbf{u}''\|_{\ell_{\eps}^{2}}^2
\end{split}
\end{align}
and since
\begin{equation}\label{4}
\sum_{\ell=-N+1}^{N}4|u''_{\ell}|^2\ge
\sum_{\ell=-N+1}^{N}\left[4|u''_{\ell}|^2-|u''_{\ell+1}-u''_{\ell}|^2\right]=
\sum_{\ell=-N+1}^{N}\left[2|u''_{\ell}|^2+2u''_{\ell+1}u''_{\ell}\right]\ge 0,
\end{equation}
the upper bound of $\la\delta^2\mathcal{E}^{a}(\mathbf{y}_{F})\mathbf{u},\mathbf{u}\ra$ is
\begin{equation}\label{up}
\la\delta^2\mathcal{E}^{a}(\mathbf{y}_{F})\mathbf{u},\mathbf{u}\ra\le
\big(\phi''_{F}+4\phi''_{2F}+9\phi''_{3F}\big)\|\mathbf{u}'\|_{\ell_{\eps}^{2}}^2
-\eps^2\big(\phi''_{2F}+6\phi''_{3F}\big)\|\mathbf{u}''\|_{\ell_{\eps}^{2}}^2.
\end{equation}

We thus define
\begin{align}\label{AfandMu1}
\begin{split}
&A_{F}^{(3)}:=\phi''_{F}+4\phi''_{2F}+9\phi''_{3F},\quad
\mu_{\eps}:=\inf_{\Psi \in \mathcal{U} \setminus \{\mathbf{0}\} }
\frac{\|\Psi''\|_{\ell_{\eps}^2}}{\|\Psi'\|_{\ell_{\eps}^{2}}}
=\frac{2\sin(\pi \eps/2)}{\eps},
\end{split}
\end{align}
where the equality $\mu_{\eps} = \frac{2\sin(\pi \eps/2)}{\eps}$
is given in \cite{doblusort:qce.stab}. We then obtain the following
stability result for the atomistic model.
\begin{theorem}\label{sharpStableAtom3rd}
Suppose $\phi''_{2F} \le 0$, $\phi''_{3F}\le 0$ and $\phi''_{kF}=0$ for $k\ge4$. Then
$\mathbf{y}_{F}$ is a stable equilibrium of the atomistic model if and only if
$A_{F}^{(3)}-\eps^2\mu_{\eps}^2\big(\phi''_{2F}+\eta_{F}\phi''_{3F}\big)
 > 0,$ where $\eta_{F}$ is a constant $2\le \eta_F\le 6$ that is uniquely determined by the deformation gradient $F,$ and
$A_{F}^{(3)}$ and $\mu_{\eps}$ are defined in
 \eqref{AfandMu1}.
 \end{theorem}
 \begin{remark}\label{AssumpPhiSecDerivative}
 The hypotheses $\phi''_{2F}\le 0,\,\phi''_{3F}\le 0$ in Theorem~\ref{sharpStableAtom3rd}
 are reasonable according to the assumptions of the interaction potential $\phi$ in the
 Section \ref{lqc}. We note that $y'_{\ell}\le \frac{r^{*}}{2}$ only occurs under extreme compression,
 and in that case the finite range pair interaction model \eqref{AtomEner3rd} itself can be expected to be invalid.
 We refer to Section~2.3 of~\cite{ortner:qnl1d} for further
discussion of this point.
 \end{remark}
\space
\subsection{The QNL Model with Third Nearest-Neighbor
Interaction Range} The total energy for the QNL model given by a
third nearest-neighbor interaction model without external forces is
\begin{align}\label{QNLenergy3rd}
\begin{split}
\mathcal{E}^{qnl}(\mathbf{y})&=\eps \sum_{\ell=-N+1}^{N}
\phi(y'_{\ell})+\eps \sum_{\ell \in A_{qnl}(2)}\phi \left(
y'_{\ell}+y'_{\ell+1}\right) + \eps \sum_{\ell \in C_{qnl}(2)} \half
\big\{\phi(2y'_{\ell})+\phi(2y'_{\ell+1})\big\}\\
&\quad + \eps \sum_{\ell \in A_{qnl}(3)}\phi \left(
y'_{\ell}+y'_{\ell+1}+y'_{\ell+2}\right) + \eps \sum_{\ell \in
C_{qnl}(3)} \frac{1}{3}
\big\{\phi(3y'_{\ell})+\phi(3y'_{\ell+1})+\phi(3y'_{\ell+2})\big\}.
\end{split}
\end{align}
%where
%\begin{align*}
%&A_{qnl}(2)=\{-K-1,-K,\dots,K,K+1\},\quad
%A_{qnl}(3)=\{-K-2,-K-1,\dots,K,K+1\},\\
%& C_{qnl}(2)=\{-N+1,\dots, N\} \setminus A_{qnl}(2),\quad \text{and}
%\quad
% C_{qnl}(3)=\{-N+1,\dots, N\} \setminus A_{qnl}(3).
%\end{align*}
Since the QNL model does not have ghost forces, the uniform
deformation $\mathbf{y}_{F}$ is still an equilibrium for
\eqref{QNLenergy3rd}. We thus focus on
$\mathcal{E}''_{qnl}(\mathbf{y}_{F})[\mathbf{u},\mathbf{u}]$ to
analyze the stability of the QNL model. We note that
$\mathcal{E}''_{qnl}(\mathbf{y}_{F})[\mathbf{u},\mathbf{u}]$ can be
written as
\begin{align}\label{SharpStabQNL1}
\begin{split}
\la\delta^{2}\mathcal{E}^{qnl}(\mathbf{y}_{F})\mathbf{u},\mathbf{u}\ra&= \eps
\sum_{\ell=-N+1}^{N} \phi''_{F}|u'_{\ell}|^2 \\
&\quad+ \eps\sum_{\ell \in A_{qnl}(2)}
\phi''_{2F}|u'_{\ell}+u'_{\ell+1}|^2 + \eps \sum_{\ell \in
C_{qnl}(2)} 4\phi''_{2F}\left\{ \half
|u'_{\ell}|^2+\half|u'_{\ell+1}|^2\right\}\\
&\qquad + \eps\sum_{\ell \in A_{qnl}(3)}
\phi''_{3F}|u'_{\ell}+u'_{\ell+1}+u'_{\ell+2}|^2 \\
&\quad+ \eps \sum_{\ell \in C_{qnl}(3)} 9\phi''_{3F}\left\{
\frac{1}{3}
|u'_{\ell}|^2+\frac{1}{3}|u'_{\ell+1}|^2+\frac{1}{3}|u'_{\ell+2}|^2\right\}.
\end{split}
\end{align}
Applying \eqref{Equa1} to \eqref{SharpStabQNL1}, we obtain
\begin{align}\label{SharpStabQNL2}
\begin{split}
\la\delta^2\mathcal{E}^{qnl}(\mathbf{y}_{F})\mathbf{u},\mathbf{u}\ra &=
A_{F}^{(3)} \|\mathbf{u}'\|_{\ell_{\eps}^2}^2-\eps^3\sum_{\ell \in
A_{qnl}(2)}\phi''_{2F}|u''_{\ell+1}|^2 \\
&\qquad-\eps^3\sum_{\ell\in A_{qnl}(3)} \phi''_{3F}
\left\{|u''_{\ell+1}|^2+
|u''_{\ell+2}|^2+|u''_{\ell+2}+u''_{\ell+1}|^2\right\}\\
%&\ge A_{F} \|\mathbf{u}'\|_{\ell_{\eps}^2}^2-\eps^3\sum_{\ell \in
%A_{qnl}(2)}\phi''_{2F}|u''_{\ell}|^2 \\
%&\quad-\eps^3\sum_{\ell\in A_{qnl}(3)} \phi''_{3F}
%\left\{3|u''_{\ell}|^2+ 3|u''_{\ell+1}|^2\right\}\\
&\ge A_{F}^{(3)} \|\mathbf{u}'\|_{\ell_{\eps}^2}^2.
\end{split}
\end{align}
Since the lower bound in \eqref{SharpStabQNL2}
is achieved by any displacement supported in the local region
(which exists unless $K \in \{N-2,N-1,N\}$),
it follows  that $\mathbf{y}_{F}$
is stable in the generalized QNL model if and only if $A^{(3)}_{F}
>0.$

\begin{theorem}\label{SharpStabQNL3rdTh}
Suppose $\phi''_{2F} \le 0$, $\phi''_{3F}\le 0$ and $\phi''_{kF}=0$ for $k\ge4$.
 Then $\mathbf{y}_{F}$ is a stable equilibrium of the QNL model if and only if
$A_{F}^{(3)}>0$.
\end{theorem}
We have given above a stability analyses of the atomistic model
and the QNL model with a third-nearest neighbor interaction range.
We now study the general case.
\subsection{The Atomistic Model with $s$th Nearest-Neighbor
Interaction Range} In the case of $s$th nearest neighbor
interactions, the total energy for the atomistic model without
external forces is
\begin{equation}\label{AtomEnerGen}
 \mathcal{E}^{a}(\mathbf{y})=\eps \sum_{\ell=-N+1}^{N}\sum_{k=1}^{s} \phi\left( \sum_{j=0}^{k-1}y'_{\ell+j}\right).
\end{equation}
The uniform deformation $\mathbf{y}_{F}$ is still its equilibrium, so the stability condition is
\begin{align}\label{AtomSharpStabGen}
\la\delta^2 \mathcal{E}^{a}(\mathbf{y}_{F})\mathbf{u},\mathbf{u}\ra=\eps
\sum_{\ell=-N+1}^{N}\sum_{k=1}^{s} \phi''_{kF}\left( \sum_{j=0}^{k-1}u'_{\ell+j}\right)^2 > 0
\qquad
\forall \mathbf{u}\in \mathcal{U}\setminus \{\mathbf{0}\}.
\end{align}

We assume that $\phi''_{F} > 0$ and
$\phi''_{kF} \le 0$ for $k=2,\dots,s.$ We first consider
$\phi''_{kF} \left(\sum_{j=0}^{k-1}u'_{\ell+j}\right)^2$ for $k\ge
2$:
\begin{align}\label{AtomShaGenEq1}
\begin{split}
 \phi''_{kF} \left(\sum_{j=0}^{k-1}u'_{\ell+j}\right)^2 &=
\phi''_{kF}\sum_{j=0}^{k-1}\left(u'_{\ell+j}\right)^2+ \phi''_{kF}\sum_{j=0}^{k-2}\sum_{i=j+1}^{k-1}2 u'_{\ell+j}\,u'_{\ell+i}\\
&=\phi''_{kF}\sum_{j=0}^{k-1}\left(u'_{\ell+j}\right)^2
 + \phi''_{kF}\sum_{j=0}^{k-2}\sum_{i=j+1}^{k-1}\left\{ \big(u'_{\ell+j}\big)^2+\big(u'_{\ell+i}\big)^2-\left(u'_{\ell+i}-u'_{\ell+j}\right)^2 \right\}.
\end{split}
\end{align}
Recalling that $u''_{\ell}:=\frac{u'_{\ell}-u'_{\ell-1}}{\eps}$,
we can further simplify \eqref{AtomShaGenEq1} and get
\begin{align}\label{AtomShaGenEq2}
\begin{split}
 \phi''_{kF} \left(\sum_{j=0}^{k-1}u'_{\ell+j}\right)^2
&=\phi''_{kF}\sum_{j=0}^{k-1}\left(u'_{\ell+j}\right)^2
 + \phi''_{kF}\sum_{j=0}^{k-2}\sum_{i=j+1}^{k-1}\left\{ \big(u'_{\ell+j}\big)^2+\big(u'_{\ell+i}\big)^2\right\}\\
&\quad -\phi''_{kF}\eps^2\sum_{j=0}^{k-2}\sum_{i=j+1}^{k-1}\left( \sum_{t=j+1}^{i}u''_{\ell+t}\right)^2 \\
& \ge \phi''_{kF}\sum_{j=0}^{k-1}\left(u'_{\ell+j}\right)^2
 + \phi''_{kF}\sum_{j=0}^{k-2}\sum_{i=j+1}^{k-1}\left\{ \big(u'_{\ell+j}\big)^2+\big(u'_{\ell+i}\big)^2\right\}\\
&\quad -\phi''_{kF}\eps^2\sum_{j=0}^{k-2}\left( u''_{\ell+j+1}\right)^2.
\end{split}
\end{align}
The last inequality comes from  the fact that
$-\phi''_{kF}\eps^2 \ge 0$ for $k=2,\dots, s$, so we determine that the
terms $i=j+2,\dots, k-1$ in the third expression above are all
nonnegative. Therefore, \eqref{AtomSharpStabGen} becomes
\begin{align}\label{AtomShaGenEq3}
\begin{split}
\la\delta^2  \mathcal{E}^{a}(\mathbf{y}_{F})\mathbf{u},\mathbf{u}\ra
&\ge \eps \sum_{\ell=-N+1}^{N}\phi''_{F} \big(u'_{\ell}\big)^2\\
&\quad   + \sum_{k=2}^{s}\eps\sum_{\ell=-N+1}^{N}\phi''_{kF}
 \left\{   \sum_{j=0}^{k-1}\big(u'_{\ell+j}\big)^2+ \sum_{j=0}^{k-2}\sum_{i=j+1}^{k-1}\left[ \big(u'_{\ell+j}\big)^2+\big(u'_{\ell+i}\big)^2\right]
-\eps^2 \sum_{j=0}^{k-2}\left( u''_{\ell+j+1}\right)^2   \right\} \\
& = \sum_{k=1}^{s} k^2 \phi''_{kF}\|u'\|_{\ell_{\eps}^{2}}^{2}
-\sum_{k=2}^{s}\eps^2\phi''_{kF}(k-1) \|u''\|_{\ell_{\eps}^{2}}^{2}.
\end{split}
\end{align}
On the other hand,
we can further study the third term in the fist line of
\eqref{AtomShaGenEq2} and rewrite it as:
\begin{align*}%\label{AtomShaGenEq4}
%\begin{split}
 -\phi''_{kF}\eps^2\sum_{j=0}^{k-2}\sum_{i=j+1}^{k-1}\left( \sum_{t=j+1}^{i}u''_{\ell+t}\right)^2
&= -\phi''_{kF}\eps^2\sum_{j=0}^{k-2}\sum_{i=j+1}^{k-1}\left[
\sum_{t=j+1}^{i}{u''_{\ell+t}}^2+\sum_{t=j+1}^{i-1}\sum_{s=t+1}^{i}\left({u''_{\ell+t}}^2+{u''_{\ell+s}}^2\right)\right]\\
&\qquad\qquad
+\phi''_{kF}\epsilon^2\sum_{j=0}^{k-2}\sum_{i=j+1}^{k-1}\sum_{t=j+1}^{i-1}\sum_{s=t+1}^{i}\left(u''_{\ell+t}-u''_{\ell+s}\right)^2.
%\end{split}
\end{align*}
Thus, we can use the first line of \eqref{AtomShaGenEq2} and $\phi''_{kF}\le 0$, $k=2,\dots,s$ to
 obtain an upper bound of $\phi''_{kF}\left(\sum_{j=0}^{k-1}u'_{\ell+j}\right)^2$:
\begin{align}\label{AtomShaGenEq4}
\begin{split}
\phi''_{kF} \left(\sum_{j=0}^{k-1}u'_{\ell+j}\right)^2&\le \phi''_{kF}\sum_{j=0}^{k-1}\left(u'_{\ell+j}\right)^2
 + \phi''_{kF}\sum_{j=0}^{k-2}\sum_{i=j+1}^{k-1}\left\{ \big(u'_{\ell+j}\big)^2+\big(u'_{\ell+i}\big)^2\right\}\\
&\qquad -\phi''_{kF}\eps^2\sum_{j=0}^{k-2}\sum_{i=j+1}^{k-1}\left[
\sum_{t=j+1}^{i}{u''_{\ell+t}}^2+\sum_{t=j+1}^{i-1}\sum_{s=t+1}^{i}\left({u''_{\ell+t}}^2+{u''_{\ell+s}}^2\right)\right].
\end{split}
\end{align}

Observing the last term of \eqref{AtomShaGenEq4} and because of the
periodicity condition of $u''_{\ell}$, for each fixed $k$, we have
that $ \sum_{j=0}^{k-2}\sum_{i=j+1}^{k-1}\left[
\sum_{t=j+1}^{i}{u''_{\ell+t}}^2+\sum_{t=j+1}^{i-1}\sum_{s=t+1}^{i}\left({u''_{\ell+t}}^2+{u''_{\ell+s}}^2\right)\right]$
is equivalent to $\sum_{j=0}^{k-2}\sum_{i=j+1}^{k-1} (i-j)^2
|u''_{\ell}|^2=\frac{k^4-k^2}{12}|u''_{\ell}|^2$.
Thus, we can obtain the following upper bound of
$\mathcal{E}''_{a}(\mathbf{y}_{F})[\mathbf{u},\mathbf{u}]$:
\begin{align}\label{AtomShaGenEq5}
\begin{split}
\la\delta^2\mathcal{E}^{a}(\mathbf{y}_{F})\mathbf{u},\mathbf{u}\ra &\le
\sum_{k=1}^{s} k^2
\phi''_{kF}\|u'\|_{\ell_{\eps}^{2}}^{2}-\sum_{k=2}^{s}\eps^2\phi''_{kF}\sum_{j=0}^{k-1}\sum_{i=j+1}^{k-1}(i-j)^2
\sum_{\ell=-N+1}^{N}\epsilon|u''_{\ell}|^{2}\\
&= \sum_{k=1}^{s} k^2 \phi''_{kF}\|u'\|_{\ell_{\eps}^{2}}^{2}
-\sum_{k=2}^{s}\eps^2\phi''_{kF}\frac{k^4-k^2}{12}
\|u''\|_{\ell_{\eps}^{2}}^{2}.
\end{split}
\end{align}
Note that when $k=2$, we have $\frac{k^4-k^2}{12}=1=k-1$,
which means that the upper bound equals the lower bound if $s=2.$

%%%%%%%%%%%%%
 We next define
\begin{equation}\label{AfGen}
A^{s}_{F}:=\sum_{k=1}^{s} k^2\phi''_{kF},
\end{equation}
and use the same notation for $\mu_{\eps}$ defined in \eqref{AfandMu1}.
We then obtain the following sharp stability estimate:
\begin{theorem}\label{sharpAtomGenStabThm}
Suppose $\phi''_{kF} \le 0$ for $k=2,\dots,s$. There exists a
constant $B=B_{F}$ satisfying
\begin{equation*}
\sum_{k=2}^{s}(k-1)\phi''_{kF}\ge B_{F}\ge
\phi''_{2F}+\sum_{k=3}^{s}\frac{k^4-k^2}{12}\phi''_{kF},
\end{equation*}
 such that
$\mathbf{y}_{F}$ is a stable equilibrium of the atomistic model if
and only if $A^{s}_{F}-\eps^2\mu_{\eps}^2B_{F}>0$, where $A^{s}_{F}$
is defined in \eqref{AfGen} and $\mu_{\eps}$ is defined in
 \eqref{AfandMu1}.
 \end{theorem}
\begin{remark}\label{AssumpPhiSecDerivative2}
The assumptions $\phi''_{kF} \le 0$ for $k=2,\dots,s$ in Theorem~\ref{sharpAtomGenStabThm}
are reasonable by the considerations given in
Remark~\ref{AssumpPhiSecDerivative}.
\end{remark}
%\space
\subsection{The QNL Model with $s$th Nearest-Neighbor
Interaction Range} We now consider the stability of the QNL model in
the general case. Again $\mathbf{y}_{F}$ is an equilibrium of the
QNL model \eqref{QNLSolEq001} when there is no external force, so we
need to check whether
\[
\la\delta^2\mathcal{E}^{qnl}(\mathbf{y}_{F})\mathbf{u},\mathbf{u}\ra>0, \quad
\forall \mathbf{u}\in\mathcal{U}\setminus \{\mathbf{0}\},
\] where
$\la\delta^2\mathcal{E}^{qnl}(\mathbf{y}_{F})\mathbf{u},\mathbf{u}\ra$ is
%that for $\mathbf{u}\in \mathcal{U}\setminus \{\mathbf{0}\}$ we have
\begin{align}\label{QNLShaGenEq1}
 \begin{split}
\la\delta^2  \mathcal{E}^{qnl}({\mathbf{y}_{F}})\mathbf{u},\mathbf{u}\ra
&= \eps \sum_{\ell=-N+1}^{N} \phi''_{F}(u'_{\ell})^2
+ \eps\sum_{k=2}^{s}\sum_{\ell \in A_{qnl}(k)} \phi''_{kF}\left(\sum_{j=0}^{k-1}u'_{\ell+j}\right)^2\\
&\quad+\eps\sum_{k=2}^{s}\sum_{\ell \in C_{qnl}(k)} k\sum_{j=0}^{k-1}\phi''_{kF}\left(u'_{\ell+j}\right)^2\\
&= \eps \sum_{\ell=-N+1}^{N} \phi''_{F}(u'_{\ell})^2 \\
&\quad+ \eps\sum_{k=2}^{s}\sum_{\ell \in A_{qnl}(k)} \phi''_{kF}\left\{\sum_{j=0}^{k-1}\big(u'_{\ell+j}\big)^2
+\sum_{j=0}^{k-2}\sum_{i=j+1}^{k-1}\left[ \big(u'_{\ell+j}\big)^2+ \big(u'_{\ell+i}\big)^2\right] \right.\\
&\qquad\quad
\left.-\sum_{j=0}^{k-2}\sum_{i=j+1}^{k-1}\eps^2\left(\sum_{t=j+1}^{i}u''_{\ell+t}\right)^2\right\}
+\eps\sum_{k=2}^{s}\sum_{\ell \in C_{qnl}(k)}
k\sum_{j=0}^{k-1}\phi''_{kF}\left(u'_{\ell+j}\right)^2,
 \end{split}
\end{align}
and for $k\ge 2$
\begin{align*}
A_{qnl}(k)&=\{-K-k+1,-K-k+2,\dots,K,K+1\},\\
C_{qnl}(k)
&=\left\{ -N+1,\dots,-K-k\right\}\bigcup\left\{ K+2,\dots,N \right\}.
\end{align*}

Since $\phi''_{kF}\le 0$ when $k \ge 2$,
\eqref{QNLShaGenEq1} becomes
\begin{align}\label{QNLShaGenEq2}
\begin{split}
\la\delta^2  \mathcal{E}^{qnl}({\mathbf{y}_{F}})\mathbf{u},\mathbf{u}\ra
&= \eps \sum_{\ell=-N+1}^{N} \phi''_{F}(u'_{\ell})^2
+ \eps\sum_{k=2}^{s}\sum_{\ell \in A_{qnl}(k)} \phi''_{kF}\left\{\sum_{j=0}^{k-1}\big(u'_{\ell+j}\big)^2\right.\\
&\quad\qquad\left.+\sum_{j=0}^{k-2}\sum_{i=j+1}^{k-1}\left[ \big(u'_{\ell+j}\big)^2+ \big(u'_{\ell+i}\big)^2\right]
 -\sum_{j=0}^{k-2}\sum_{i=j+1}^{k-1}\eps^2\left(\sum_{t=j+1}^{i}u''_{\ell+t}\right)^2\right\}\\
&\qquad+\eps\sum_{k=2}^{s}\sum_{\ell \in C_{qnl}(k)} k\sum_{j=0}^{k-1}\phi''_{kF}\left(u'_{\ell+j}\right)^2\\
&= A^{s}_{F} \|u'\|_{\ell_{\eps}^{2}}^{2}
-\eps\sum_{k=2}^{s}\eps^2\phi''_{kF}\sum_{\ell \in A_{qnl}(k)}
\sum_{j=0}^{k-2}\sum_{i=j+1}^{k-1}\eps^2\left(\sum_{t=j+1}^{i}u''_{\ell+t}\right)^2\\
&\ge A^{s}_{F} \|u'\|_{\ell_{\eps}^{2}}^{2},
\end{split}
\end{align}
where $A^{s}_{F}$ is defined in \eqref{AfGen}. Since the lower bound in \eqref{QNLShaGenEq2}
is achieved by any displacement supported in the local region
(which exists unless $K \in \{N-s+1,\dots,N\}$),
it follows  that $\mathbf{y}_{F}$
is stable in the QNL model if and only if $A^{s}_{F}
>0.$
\begin{theorem}\label{SharpQNLStabThm}
Suppose that $K <N-s+1$ and that $\phi''_{kF} \le 0$ for $k=2,\dots,s.$
Then $\mathbf{y}_{F}$ is a stable equilibrium of the QNL model if and only if
$A^{s}_{F}>0$.
\end{theorem}
%%%%%%%%%%%%%%5
\begin{remark}
From Theorem \ref{sharpAtomGenStabThm} and Theorem
\ref{SharpQNLStabThm}, we conclude that the difference between the
sharp stability conditions of the fully atomistic and the QNL models
is of order $O(\epsilon^2)$. This result is the same as for the pair
potential case in \cite{doblusort:qcf.stab}.
\end{remark}

%**************************************************************
\section{Convergence of the QNL model.}\label{convergence}
So far, we have investigated the stability of the fully
atomistic model and the QNL model.
In this section, we will give an optimal order error analysis for the
QNL model. %Since we assume that the deformation $y'_{\ell}$ is
%``smooth" in the continuum region and not in the atomistic region,
We compare the QNL solution with the atomistic solution and
give an error estimate in terms of the deformation in the
continuum region.

\subsection{The Atomistic Model with External Dead Load}

The total atomistic energy with an external dead load $\mathbf{f}$
is
\[
\mathcal{E}_{tot}^{a}(\mathbf{y}):=
\mathcal{E}^{a}(\mathbf{y})+\mathcal{F}(\mathbf{y})\qquad\forall
\mathbf{y}\in \mathcal{Y}_F,
\]
where
\[
\mathcal{F}(\mathbf{y}):=-\sum_{\ell=-N+1}^{N}\eps f_{\ell}y_{\ell}.
\]
To guarantee the existence of energy-minimizing deformations,
we assume that the external loading force $\mathbf{f}$ is in
$\mathcal{U}.$
The equilibrium solution $\mathbf{y}^{a}\in \mathcal{Y}_F$ of the
atomistic model with external force $\mathbf{f}$ then satisfies
\begin{equation}\label{AtomSolEq1}
-\la\delta\mathcal{E}^{a}(\mathbf{y}^{a}),\mathbf{w}\ra
= \la\delta\mathcal{F}(\mathbf{y}^{a}),\mathbf{w}\ra\quad
\forall \mathbf{w}\in \mathcal{U},
\end{equation}
where
\begin{equation}
\label{AtomSolEq2}
\la\delta\mathcal{E}^{a}(\mathbf{y}^{a}),\mathbf{w}\ra=\eps
\sum_{\ell=-N+1}^{N}\sum_{k=1}^{s}\sum_{j=0}^{k-1}
 \phi'\left(\sum_{t=0}^{k-1}Dy^{a}_{\ell+t}\right)w'_{\ell+j},
\end{equation}
and the external force is given by
\[
\la\delta \mathcal{F}(\mathbf{y}),\mathbf{w}\ra:=\sum_{\ell=-N+1}^{N}\frac{\partial \mathcal{F}}{\partial y_{\ell}}(\mathbf{y})w_{\ell}
=
-\sum_{\ell=-N+1}^{N}\eps
 f_{\ell}w_{\ell}.
\]
\subsection{The General QNL Model with External Dead Load}
The total energy of the QNL
model corresponding to a deformation $\mathbf{y}\in \mathcal{Y}_F$ is
\begin{align*}
\mathcal{E}_{tot}^{qnl}(\mathbf{y}):=\mathcal{E}^{qnl}(\mathbf{y})+\mathcal{F}(\mathbf{y})
%=&
%\epsilon\sum_{\ell=-N+1}^{N}\phi\left(y'_{\ell}\right)+
% \eps\sum_{k=2}^{s}\sum_{\ell \in A_{qnl}(k)} \phi\left(\sum_{j=0}^{k-1}y'_{\ell+j}\right)\\
 %&\quad +\eps\sum_{k=2}^{s}\sum_{\ell \in C_{qnl}(k)} \frac{1}{k}\sum_{j=0}^{k-1}\phi\left(ky'_{\ell+j}\right)
%-\sum_{\ell=-N+1}^{N}\eps f_{\ell}y_{\ell}.
\end{align*}
So, the equilibrium solution $\mathbf{y}^{qnl}\in \mathcal{Y}_F$ satisfies
\begin{equation}\label{Sol}
-\la\delta\mathcal{E}^{qnl}(\mathbf{y}^{qnl}),\mathbf{w}\ra=\la\delta\mathcal{F}(\mathbf{y}^{qnl}),\mathbf{w}\ra
\quad
\forall \mathbf{w}\in \mathcal{U},
\end{equation}
where
\begin{align}\label{QNLSolEq1}
\begin{split}
\la\delta\mathcal{E}^{qnl}(\mathbf{y}^{qnl}),\mathbf{w}\ra
&= \eps \sum_{\ell=-N+1}^{N} \phi'(y'_{\ell})w'_{\ell}
+\eps\sum_{k=2}^{s}\,\sum_{\ell \in A_{qnl}(k)} \phi'\left( \sum_{t=0}^{k-1}y'_{\ell+t}\right)
\left( \sum_{j=0}^{k-1}w'_{\ell+j}\right)\\
&\qquad +\eps \sum_{k=2}^{s}\,\sum_{\ell \in
C_{qnl}(k)}\sum_{j=0}^{k-1}\phi'\left(
k\,y'_{\ell+j}\right)w'_{\ell+j} \quad \forall \mathbf{w}\in
\mathcal{U}.
\end{split}
\end{align}
%%%

Setting
$\mathbf{y}^{qnl}=\mathbf{y}_{F}+\mathbf{u}^{qnl}$ and
$\mathbf{y}^{a}=\mathbf{y}_{F}+\mathbf{u}^{a}$, where both
$\mathbf{u}^{qnl}$ and $\mathbf{u}^{a}$ belong to $\mathcal{U},$
we define the quasicontinuum error to be
\[
\mathbf{e}^{qnl}:=\mathbf{y}^{a}-\mathbf{y}^{qnl}=\mathbf{u}^{a}-\mathbf{u}^{qnl}.
\]
To simplify the error analysis, we consider the
linearization of the atomistic equilibrium equations \eqref{AtomSolEq1} and the
associated QNL equilibrium equations \eqref{Sol}
 about the uniform deformation
$\mathbf{y}_{F}.$  The linearized atomistic equation is
\begin{equation}\label{linatom}
-\la\delta^{2}\mathcal{E}^{a}\left(\mathbf{y}_{F}\right)\mathbf{u}^{a},\mathbf{w}\ra
=\la\delta \F(\mathbf{y}_F),\mathbf{w}\ra\quad\text{for all }\mathbf{w}\in \mathcal{U},
\end{equation}
and the linearized QNL  equation is
\begin{equation}\label{linqnl}
-\la\delta^{2}\mathcal{E}^{qnl}\left(\mathbf{y}_{F}\right)\mathbf{u}^{qnl},\mathbf{w}\ra
=\la\delta \F(\mathbf{y}_F),\mathbf{w}\ra\quad\text{for all }\mathbf{w}\in \mathcal{U}.
\end{equation}
We thus analyze the linearized error equation
\begin{equation}\label{errorequation}
\la\delta^{2}\mathcal{E}^{qnl}\left(\mathbf{y}_{F}\right)\mathbf{e}^{qnl},\mathbf{w}\ra
=\la\mathbf{T}^{qnl},\mathbf{w}\ra\quad\text{for all }\mathbf{w}\in
\mathcal{U},
\end{equation}
where the linearized consistency error is given by
\begin{align}\label{QNLConvergenceEq2}
\la\mathbf{T}^{qnl},\mathbf{w}\ra&:=
\la\delta^{2}\mathcal{E}^{qnl}\left(\mathbf{y}_{F}\right)
 \mathbf{u}^{a},\mathbf{w}\ra
-\la\delta^{2}\mathcal{E}^{a}\left(\mathbf{y}_{F}\right)
\mathbf{u}^{a},\mathbf{w}\ra.
\end{align}
%%%%
%%%%%%%%%%%%%%

To obtain an optimal order consistency error estimate, we extend the
negative norm method given in \cite{dobs-qcf2,ortner:qnl1d} to the
$s$th nearest-neighbor linearized consistency error functional
\begin{align}\label{TruncationQNLEq1}
\begin{split}
\la\mathbf{T}^{qnl},\mathbf{w}\ra= \eps \sum_{k=2}^{s}\sum_{\ell \in
C_{qnl}(k)} \phi''_{kF}\sum_{j=0}^{k-1}\left(
kDu_{\ell+j}-\sum_{t=0}^{k-1}Du_{\ell+t}\right)Dw_{\ell+j},\qquad \forall \mathbf{w}\in\mathcal{U}.
\end{split}
\end{align}

For each fixed $k$, we define $\mathbf{T}^{qnl}_{k}:=\epsilon
\phi''_{kF}\sum_{\ell\in C_{qnl}(k)}\sum_{j=0}^{k-1}\left(k
Du_{\ell+j}-\sum_{t=0}^{k-1}Du_{\ell+t}\right)$. Then we have
\begin{align}\label{TruncationQNLEq2}
\begin{split}
\la\mathbf{T}^{qnl}_{k},\mathbf{w}\ra & =\eps
\phi''_{kF}\sum_{\ell=-N+1}^{-K-k}
\sum_{j=0}^{k-1}\left(kDu_{\ell+j}-\sum_{t=0}^{k-1}Du_{\ell+t}\right)Dw_{\ell+j}\\
&\quad\quad +\eps \phi''_{kF}\sum_{\ell=K+2}^{N}
\sum_{j=0}^{k-1}\left(kDu_{\ell+j}-\sum_{t=0}^{k-1}Du_{\ell+t}\right)Dw_{\ell+j}\\
&\quad =\epsilon
\phi''_{kF}\sum_{\ell=K+2}^{2N-K-k}\sum_{j=0}^{k-1}\left(kDu_{\ell+j}-\sum_{t=0}^{k-1}Du_{\ell+t}\right)Dw_{\ell+j}.
\end{split}
\end{align}
We use the $2N$ periodicity of the model in the last equality to
simplify the expression of \eqref{TruncationQNLEq2}.

Note that we can change the indices, rearrange the order of summation of the sums in the last equality of \eqref{TruncationQNLEq2},
and rewrite it as the following expression
\begin{align}\label{TruncationEq3}
\begin{split}
\la\mathbf{T}^{qnl}_{k},\mathbf{w}\ra
&=\epsilon\phi''_{kF}\sum_{\ell=K+k+1}^{2N-K-k}\sum_{j=1}^{k}(k-j)\left(-Du_{\ell-j}+2Du_{\ell}-Du_{\ell+j}\right)Dw_{\ell}\\
&\quad+\epsilon\phi''_{kF}\sum_{\ell=K+2}^{K+k}\sum_{j=0}^{\ell-(K+2)}\left(kDu_{\ell}-\sum_{t=0}^{k-1}Du_{\ell-j+t}\right)Dw_{\ell}\\
&\qquad +\epsilon\phi''_{kF}\sum_{\ell=2N-K-2k+2}^{2N-K-k}\sum_{j=2N-K-k+1-\ell}^{k-1}\left(kDu_{\ell+j}-\sum_{t=0}^{k-1}Du_{\ell+t}\right)Dw_{\ell+j}.
\end{split}
\end{align}
The first term of \eqref{TruncationEq3} corresponds to the inner continuum region, which is of second-order because of
the symmetries of the interaction. The second and the third terms are the interfacial terms. They are only of first-order since they
lose the symmetries of the interaction.

Now we will give an estimate of the consistency error $\la\mathbf{T}^{qnl},\mathbf{w}\ra$ in the following theorem.
We first define the following semi-norms:
\[
\|\mathbf{v}\|^2_{\ell_{\eps}^{2}(\mathcal{\tilde C}_{qnl}(k))}:=\epsilon\sum_{\ell\in  \mathcal{\tilde C}_{qnl}(k)}v_{\ell}^2\quad\text{and}\quad
\|\mathbf{v}\|^2_{\ell_{\eps}^{2}(\mathcal{I}_{qnl}(k))}:=\epsilon\sum_{\ell\in
\mathcal{I}_{qnl}(k)}v_{\ell}^2,
\]
where $\mathcal{I}_{qnl}(k):=\{-K-k+1,\dots,-K-1\}\bigcup\{K+2,\dots,K+k\}$ is
the interface between the continuum and atomistic regions, $
\mathcal{\tilde C}_{qnl}(k):=\mathcal{C}_{qnl}(k)\bigcup \mathcal{I}_{qnl}(k)$.
\begin{theorem}\label{TruncationError}
The consistency error $\la\mathbf{T}^{qnl},\mathbf{w}\ra,$ given in \eqref{TruncationQNLEq1}, satisfies the
following negative norm estimate
\begin{align}\label{TruncationEq4}
 \begin{split}
  \left|\la\mathbf{T}^{qnl},\mathbf{w}\ra\right|
& \le \left\{\sum_{k=2}^{s}\eps^2
\mathtt{C}_{1}(k) |\phi''_{kF}| \|D^{(3)}\mathbf{u}\|_{\ell_{\eps}^{2}(\mathcal{\tilde C}_{qnl}(k))}\right.\\
&\qquad \quad+\left.\sum_{k=2}^{s}\eps \mathtt{C}_{2}(k) |\phi''_{kF}|\,
\left(\sqrt{2s\eps}\right) \|D^{(3)}\mathbf{u}\|_{\ell_{\eps}^{\infty}( \mathcal{ I}_{qnl}(k))} \right\}\|D\mathbf{w}\|_{\ell_{\eps}^{2}},
 \end{split}
\end{align}
where $\mathtt{C}_{1}(k)$,
$\mathtt{C}_{2}(k)$ are positive constants independent of $\eps$.
\end{theorem}
\begin{proof}
From \eqref{TruncationEq3}, we have
\begin{equation}\label{TruncationQNLEq5}
\begin{split}
\left|\la\mathbf{T}^{qnl}_{k},\mathbf{w}\ra\right|
& \le
         \eps^2\mathtt{C}_{1}(k) |\phi''_{kF}|\|D^{(3)}\mathbf{u}\|_{\ell_{\eps}^{2}(\mathcal{\tilde C}_{qnl}(k))}\|D\mathbf{w}\|_{\ell^{2}_{\eps}}
          +\eps\mathtt{C}_{2}(k)
          |\phi''_{kF}|\|D^{(2)}\mathbf{u}\|_{\ell_{\eps}^{2}(\mathcal{I}_{qnl}(k))}\|D\mathbf{w}\|_{\ell^{2}_{\eps}}.
\end{split}
\end{equation}
Therefore, we obtain an optimal order
estimate for \eqref{TruncationQNLEq1}
\begin{align}\label{TruncationQNLEq6}
 \begin{split}
  \left|\la\mathbf{T}^{qnl},\mathbf{w}\ra\right|
\le \left[\sum_{k=2}^{s}\eps^2\mathtt{C}_{1}(k) |\phi''_{kF}|
\|D^{(3)}\mathbf{u}\|_{\ell_{\eps}^{2}(\mathcal{\tilde C}_{qnl}(k))}
    +\sum_{k=2}^{s}\eps \mathtt{C}_{2}(k) |\phi''_{kF}|\,\|D^{(2)}\mathbf{u}\|_{\ell_{\eps}^{2}(\mathcal{I}_{qnl}(k))}\right]
    \|D\mathbf{w}\|_{\ell_{\eps}^{2}}.
 \end{split}
\end{align}

We note that we have
\begin{align*}
 \|D^{(2)}\mathbf{u}\|_{\ell_{\eps}^{2}(\mathcal{I}_{qnl}(k))}^{2}
 &=\eps \sum_{\ell \in I(k)} \left(D^{(2)}\mathbf{u}_{\ell}\right)^2
\le \eps \|D^{(2)}\mathbf{u}\|_{\ell^{\infty}_{\eps}(\mathcal{I}_{qnl}(k))}^2 \, \sum_{\ell \in \mathcal{I}_{qnl}(k)} 1 \\
&\qquad \le \eps \, 2s\,
\|D^{(2)}\mathbf{u}\|_{\ell_{\eps}^{\infty}(\mathcal{I}_{qnl}(k))}^{2}.
%\le  4
%s\eps \|D^{(3)}\mathbf{u}\|_{\ell_{\eps}^{2}(\tilde C_{qnl}(k))}^{2}.
\end{align*}
Thus, we can obtain from \eqref{TruncationQNLEq6}
the more concise (but weaker) estimate
\begin{align}\label{TruncationQNLEq7}
 \begin{split}
  \left|\la\mathbf{T}^{qnl},\mathbf{w}\ra\right|
& \le \left\{\sum_{k=2}^{s}\eps^2
\mathtt{C}_{1}(k) |\phi''_{kF}| \|D^{(3)}\mathbf{u}\|_{\ell_{\eps}^{2}(\mathcal{\tilde C}_{qnl}(k))}\right.\\
&\qquad \quad+\left.\sum_{k=2}^{s}\eps \mathtt{C}_{2}(k) |\phi''_{kF}|\,
\left(\sqrt{2s\eps}\right) \|D^{(3)}\mathbf{u}\|_{\ell_{\eps}^{\infty}(\mathcal{I}_{qnl}(k))} \right\}\|D\mathbf{w}\|_{\ell_{\eps}^{2}}.
 \end{split}
\end{align}
\end{proof}%\quad\qed

We can finally derive the convergence rate of the generalized QNL model
from the above consistency error estimate and the stability estimate~\eqref{QNLShaGenEq2}.
\begin{theorem}\label{ConvergenceQNLTh1}
 Suppose that $A_{F}^{s} >0$, where $A_{F}^{s}$ is defined in \eqref{AfGen}.
 Then the linearized
atomistic problem \eqref{linatom} as well as the linearized QNL
approximation \eqref{linqnl} have unique solutions, and
they satisfy the error estimate
\begin{align*}
 \|D\mathbf{u}^{a}-D\mathbf{u}^{qnl}\|_{\ell_{\eps}^{2}}
 &\le
 \frac{\sum_{k=2}^{s}\eps^2
\mathtt{C}_{1}(k) |\phi''_{kF}| \|D^{(3)}\mathbf{u}\|_{\ell_{\eps}^{2}(\mathcal{\tilde C}_{qnl}(k))}}{A_{F}^s}\\
&\qquad \quad+\frac{\sum_{k=2}^{s}\eps^{3/2} \left(\sqrt{2s}\right) \mathtt{C}_{2}(k) |\phi''_{kF}|\,
 \|D^{(3)}\mathbf{u}\|_{\ell_{\eps}^{\infty}(\mathcal{I}_{qnl}(k))}}{A_{F}^{s}}.
\end{align*}
\end{theorem}
\begin{proof}
This error estimate for the generalized QNL model follows from
the error equation \eqref{errorequation}, the stability result
\eqref{QNLShaGenEq2}, and the estimate of the consistency error \eqref{TruncationQNLEq7}.
\end{proof}%\quad \qed
\begin{remark}\label{sharpConverg}
The above O($\epsilon^{3/2}$) error estimate for the strain in ${\ell_{\eps}^{2}}$ is the optimal order of convergence. For the next nearest neighbor interaction,
an explicit expression
of the linearized error is given in \cite{Dobson:2008b} that justifies this claim.
\end{remark}
\begin{remark}\label{ContRegionNorm}
We note that the error estimate in Theorem~\ref{ConvergenceQNLTh1} depends only on the smoothness
of the strain in the continuum and interfacial regions and it holds for linearizations near lattice instabilities according to the stability
estimate in Theorem~\ref{SharpQNLStabThm},
thus the generalized
QNL method can give a small error if defects are captured in the atomistic region.
\end{remark}
%\space
%\newline

 \section{Conclusion.}
We propose a generalization of the one-dimensional QNL method to allow for arbitrary
finite range interactions.
 We study the stability and convergence of a linearization of the generalized
 QNL energy with arbitrary $s$th nearest-neighbor interaction range.
We extend the methods given in
\cite{doblusort:qce.stab, doblusort:qcf.stab} to give sharp conditions
under which the atomistic model and the generalized QNL model are stable.
The difference of the stability conditions between the QNL and
  atomistic model
  is shown to be of order O($\epsilon^2$).

  We then give a negative norm estimate for the consistency error and
  generalize the conclusions in \cite{Dobson:2008b} to the
finite-range interaction case. We compare the equilibria of the generalized
QNL model and the atomistic model and give
an  optimal order O($\epsilon^{3/2}$) error
  estimate in ${\ell_{\eps}^{2}}$ for the strain that depends only on the deformation in the continuum region.
\section{Acknowledge}
We appreciate the help from Matthew Dobson and Dr. Christoph Ortner.

\end{document}